\documentclass[11pt]{amsart}

\usepackage[active]{srcltx}

 \usepackage{latexsym}
\usepackage{amssymb}
\usepackage{graphics}

\newtheorem{theorem}{Theorem}[section]
\newtheorem{corollary}{Corollary}[section]
\newtheorem{lemma}{Lemma}[section]

\newtheorem{remark}{Remark}[section]

\newtheorem{proposition}{Proposition}[section]

\def \l{\lambda }

\title[Fusion rules and  complete reducibility]{Fusion rules and  complete reducibility of certain modules for affine Lie algebras}
\subjclass[2000]{ Primary 17B69, Secondary 17B67, 81R10}
\author{Dra\v{z}en Adamovi\'{c} and Ozren Per\v{s}e}
\address{Faculty of Science - Department of Mathematics, University of Zagreb, Croatia}
\email{adamovic@math.hr; perse@math.hr}

\begin{document}

\newcommand{\wta}{{\rm {wt} }  a }
\newcommand{\R}{\frak R}
\newcommand{\Lat}{ R}
\newcommand{\qtrp}{\overline{V_L}}

\newcommand{\trp}{\mathcal{W}(p)}

\newcommand{\wtb}{{\rm {wt} }  b }
\newcommand{\bea}{\begin{eqnarray}}
\newcommand{\eea}{\end{eqnarray}}
\newcommand{\be}{\begin {equation}}
\newcommand{\ee}{\end{equation}}
\newcommand{\g}{\frak g}
\newcommand{\hg}{\hat {\frak g} }
\newcommand{\hn}{\hat {\frak n} }
\newcommand{\h}{\frak h}
\newcommand{\V}{\Cal V}
\newcommand{\hh}{\hat {\frak h} }
\newcommand{\n}{\frak n}
\newcommand{\Z}{\Bbb Z}
\newcommand{\N}{{\Bbb Z} _{> 0} }
\newcommand{\Zp} {\Z _ {\ge 0} }

\newcommand{\C}{\Bbb C}
\newcommand{\Q}{\Bbb Q}
\newcommand{\1}{\bf 1}
\newcommand{\la}{\langle}
\newcommand{\ra}{\rangle}

\newcommand{\wt}{{\rm {wt} }   }

\newcommand{\cA}{\Cal A}
\newcommand{\non}{\nonumber}
\def \l {\lambda}
\baselineskip=14pt
\newenvironment{demo}[1]%
{\vskip-\lastskip\medskip
  \noindent
  {\em #1.}\enspace
  }%
{\qed\par\medskip
  }

\maketitle
\begin{abstract}We develop a new method for  obtaining branching rules  for affine Kac-Moody Lie algebras at negative integer levels.
This method uses fusion rules for vertex operator algebras of affine
type. We prove  that an infinite family of ordinary modules for
affine vertex  algebra of type $A$ investigated in \cite{AP} is
closed under fusion. Then we apply these fusion rules on  explicit
bosonic realization of level $-1$ modules for the
 affine Lie algebra  of type $A_{\ell-1}^{(1)}$, obtain a new proof of complete reducibility for these representations, and
 the corresponding decomposition for $\ell \ge 3$.  We also obtain the complete reducibility of the associated level
$-1$ modules for affine Lie algebra of type $C_{\ell}^{(1)}$. Next
we notice that the category of $D_{ 2 \ell -1}^{(1)}$ modules at
level $-  2 \ell +3 $ obtained in \cite{P2}  has the isomorphic fusion
algebra. This enables us to  decompose certain  $E_6 ^{(1)}$ and
$F_4 ^{(1)}$--modules at negative levels.
\end{abstract}

\section{Introduction}

In the present paper we study  branching rules   for certain modules
for affine Lie algebras at negative levels. We study the free field
realizations of modules for affine Lie algebra of type $A$ of level
$-1$ introduced in \cite{FF} and certain representations which
naturally appear in the context of conformal embeddings in our
recent paper  \cite{AP-ART}.

V. Kac and M. Wakimoto considered in  \cite{KW-cmp-2001} the free
field representations of the affine Lie superalgebra
$\widehat{\mathfrak{gl}(m \vert \ell) }$  realized on the tensor product
$F_{m} \otimes M_{\ell}$, where $F_m$ (resp. $M_{\ell}$) is a vertex
algebra associated to infinite-dimensional Clifford (resp. Weyl)
algebra. They proved that every charge component $(F_{m} \otimes
M_{\ell}) ^{(s)}$ is an irreducible $\widehat{\mathfrak{gl}(m \vert
\ell)} $--module if $m \ge 1$. Their proof can not be applied in the
case $m=0$ since it uses super boson-fermion correspondence. Such
approach requires at least one fermionic field, and it is a
generalization of super boson-fermion correspondence based on
$\widehat{\mathfrak{gl}(1 \vert 1) }$ from \cite{KL}. So for the
case $m=0$, one needs different approach. V. Kac and M. Wakimoto
also noted in Remark 3.3. of \cite{KW-cmp-2001} that $M_{2} ^{(s)}$
is not irreducible $\widehat{\mathfrak{gl}(2) }$--module. So this
makes the problem of irreducibility of $M_{\ell} ^{(s)}$ open for
$\ell \ge 3$.

In this paper we extend the  Kac-Wakimoto irreducibility
result to $\widehat{\mathfrak{gl} (0, \ell)}
=\widehat{\mathfrak{gl}(\ell) }$--modules $M_{\ell} ^{(s)}$ for
$\ell \ge 3$. We will prove that every charge component $M_{\ell}
^{(s)}$ is an irreducible $\widehat{\mathfrak{gl}(\ell)} $--module.

Let us describe our construction in more detail.

Vertex operator algebra $M_{\ell}$ is simple and admits the charge
decomposition $M_{\ell}= \oplus _{s \in \Z} M_{\ell}^{(s)}$. The
charge zero subspace $M_{\ell}^{(0)}$ is a simple vertex operator
subalgebra of $M_{\ell}$ which contains the vertex operator algebra
$\widetilde{L}_{A_{ \ell -1} ^{(1)}} (-\Lambda_0)$ (defined by
relation (\ref{L-tilda})) associated to affine Lie algebra of type
$A_{ \ell -1} ^{(1)}$ of level $-1$, and a copy of rank one
Heisenberg vertex algebra $M(1)$, commuting with $\widetilde{L}_{A_{
\ell -1} ^{(1)}} (-\Lambda_0)$.

Let $\ell \ge 3$. In this paper we present a vertex algebraic proof
of the isomorphism
\bea && M_{\ell}^{(0)} \cong \widetilde{L}_{A_{ \ell -1} ^{(1)}}
(-\Lambda_0) \otimes M(1) \label{charge-0} \eea
which proves that $ \widetilde{L}_{A_{ \ell -1} ^{(1)}} (-\Lambda_0)$ is simple.

Our proof is based  the classification results from \cite{AP}, the notion
of fusion rules for vertex algebra modules and the fact that
$M_{\ell}$ is a simple vertex operator algebra. It turns out that ${A_{ \ell -1} ^{(1)}} $--modules
which are realized inside of $M_{\ell}$ are modules for the vertex operator algebra $V_{\ell-1,1}$
investigated in \cite{AP}, which  enables us to prove that these modules are closed under fusion.
Then explicit information on fusion rules and standard techniques in the vertex algebra theory prove (\ref{charge-0}).
We also show that
every $M_{\ell} ^{(s)}$ ($s \in \Z$) is an irreducible ${L}_{A_{\ell
-1} ^{(1)}}(-\Lambda_0) \otimes M(1)$--module.

It is interesting to notice that (\ref{charge-0}) does not hold if
$\ell \le 2$. For $\ell =1$, $M_{\ell} ^{(0)}$ is isomorphic to the
vertex algebra $\mathcal{W}_{1+ \infty}$ at central charge $c=-1$
(cf. \cite{KR}, \cite{Ad-com-alg}, \cite{Lin}, \cite{Wang-cmp}). The
case $\ell =2$ was described in \cite{KW-cmp-2001}.

Next, we consider some vertex subalgebras of $L_{A_{2 \ell -1}
^{(1)}} (-\Lambda_0)$ of affine type. In this  paper we shall
consider the vertex  subalgebra of type $C$, and the associated
bosonic representation $M_{2 \ell}$. The results from \cite{AP2}
then imply that $M_{2 \ell}$ is a completely reducible $L_{ C_{\ell}
^{(1)} } (-\Lambda_0)$--module.

In this paper we also find  that there can exist vertex operator
algebras of affine type having the same fusion algebras as
$L_{A_{\ell-1} ^{(1)}} (- \Lambda_0)$. For that purpose, we study
fusion rules for vertex operator algebras  of type $D$ introduced
recently  in \cite{P2}. We are focused on the category of ordinary
modules for vertex operator algebra ${\mathcal V}_{D_{2 \ell-1}
^{(1)}} (- (2 \ell -3) \Lambda_0)$ which is a non-trivial quotient
of the universal vertex operator algebra of level $-2 \ell +3$. In
order to determine the fusion rules, one needs to decompose certain
tensor products of irreducible finite-dimensional representations of
type $D$. But this job was done by S. Okada in \cite{Okada}. The
required tensor product is multiplicity free, and by using his
result we are able to calculate fusion rules for the  irreducible
representations classified in \cite{P2}. These fusion rules have
many interesting applications.

We first classify irreducible representations for the simple vertex
operator algebra $L_{D_{2 \ell-1} ^{(1)}} (- (2 \ell -3)
\Lambda_0)$. Next we are focused on the case $\ell =3$. Then the
simple vertex operator algebra $L_{D_{5} ^{(1)}} ( -3 \Lambda_0)$
can be realized as a subalgebra of the simple vertex operator
algebra $L_{ E_6 ^{(1)} }(- 3 \Lambda_0)$. By combining our fusion
rules analysis, we are able to decompose the vertex operator algebra
$L_{ E_6 ^{(1)} } (- 3 \Lambda_0)$ as a direct sum of irreducible
$L_{D_{5} ^{(1)}} ( -3 \Lambda_0) \otimes M(1)$--modules, where
$M(1)$ is the Heisenberg vertex operator algebra with central
charge~$1$. We obtain certain similar decompositions which follow
from the results on conformal embeddings from \cite{AP-ART}.

We assume that the reader is familiar with the notion of vertex
operator algebra (cf. \cite{B}, \cite{FHL}, \cite{FLM}, \cite{FrB},
\cite{FZ}, \cite{K2}, \cite{L}, \cite{LL}) and Kac-Moody algebra
(cf. \cite{K1}). Throughout the paper, we use the following
notation: ${\frak g}_{X_{\ell}}$ for simple Lie algebra of type
$X_{\ell}$, and $\hat{\frak g}_{X_{\ell}}$ for the associated affine
Lie algebra of type $X_{\ell} ^{(1)}$, $V_{X_{\ell}}(\mu)$ for
irreducible highest weight module for ${\frak g}_{X_{\ell}}$,
$L_{X_{\ell} ^{(1)}}(\lambda)$ for irreducible highest weight module
for $\hat{\frak g}_{X_{\ell}}$ and $N_{X_{\ell} ^{(1)}}(k
\Lambda_0)$ for the universal affine vertex algebra of level $k$
associated to $\hat{\frak g}_{X_{\ell}}$.

\bigskip
{\bf Acknowledgments:} We would like to thank the referee for his valuable comments.

\section{Weyl vertex algebras and some affine Lie algebras}

\label{sect-weyl-vertex}

The Weyl algebra $W_{\ell}( \tfrac{1}{2} + {\Z})$ is a complex
associative algebra generated by
$$  a ^{\pm} _{i}(r),  \quad   r  \in \tfrac{1}{2} + {\Z}, \ 1 \le i \le \ell$$ with non-trivial relations
$$
 [a_{i} ^{+}(r)  , a_{j} ^{-}(s)  ]  =\delta_{r+s,0} \delta_{i,j}
$$
where $r, s \in {\tfrac{1}{2}}+ {\Z}$, $ i, j \in \{1, \dots, \ell
\}$.

 Let $M_{\ell}$ be the irreducible $W_{\ell}( \tfrac{1}{2} + {\Z})$--module generated by
 the
cyclic vector ${\1}$ such that

$$ a^{\pm} _i  (r) {\1}     =  0 \quad
\mbox{for} \ \ r > 0 ,  \ 1 \le i \le \ell .$$

Define the following   fields on $M_{\ell}$
$$   a_i ^{\pm }(z) = \sum_{ n \in   {\Z}
 } a_i ^{\pm }(n+ \tfrac{1}{2} )  z ^{-n- 1}, \quad 1 \le i \le \ell. $$
 The fields $ a_{i} ^{\pm }(z)$, $i =1, \dots , \ell $ generate on $M_{\ell}$  the
unique structure of a simple vertex algebra (cf. \cite{K2},
\cite{FrB}). Let us denote the corresponding vertex operator by $Y$.

We have the following conformal vector in $M_{\ell}$:
\bea &&\omega=\frac{1}{2} \sum_{i=1} ^{\ell} ( a_i ^- (-
\frac{3}{2}) a_i ^{+} (-\frac{1}{2} ) -a_i ^+ (-  \frac{3}{2}) a_i
^{-} (-\frac{1}{2} ) ) {\1}. \label{vir-bos} \eea
Let $Y(\omega,z) = \sum_{n \in {\Z} } L(n) z ^{-n-2}$.
Then $M_{\ell}$ is a $\tfrac{1}{2}{\Zp}$--graded with respect to
$L(0)$:
$$ M_{\ell} = \bigoplus_{ m \in \tfrac{1}{2} \Zp} M_{\ell} (m), \quad M_{\ell} (m)= \{ v \in M_{\ell} \ \vert \  L(0) v = m v \}.$$
Note that $M_{\ell} (0) = {\C} {\1}$. For $ v \in M_{\ell} (m)$ we
shall write $\wt (v) = m$.

Define

$$ H = - \sum_{i = 1 } ^{\ell} a _i ^+ (-\tfrac{1}{2}) a_i ^- (-\tfrac{1}{2}) {\bf 1}. $$

Let $M(1)$ be the Heisenberg vertex subalgebra of $M_{\ell}$ generated by $H$.
Let $Y(H, z) = \sum_{n \in {\Z}} H(n) z ^{-n-1}.$

Then
$$[H(0), a_i^{\pm} (n)] =\pm a_i^{\pm} (n), \quad 1 \le i \le \ell. $$
Clearly, $H(0)$ acts semisimply on $M_{\ell}$ and it defines the
following ${\Z}$--gradation, called the charge decomposition:
$$
M_{\ell} = \bigoplus_{ s \in \Z} M_{\ell} ^{(s)} , \quad M_{\ell} ^{(s)} = \{ v \in M_{\ell} \ \vert \ H(0) v =  s v \}. $$
Denote by $M(1, r)$ the irreducible $M(1)$--module on which $H(0)$
acts as $r \mbox{Id}$, $r \in {\C}$.

The subalgebra of $M_{\ell}$ generated by
\bea && e_{\epsilon _{i}- \epsilon _{j}} = a_i ^+ (-\tfrac{1}{2})
a_{j} ^-( -\tfrac{1}{2})
 {\1}, \quad f_{\epsilon _{i}- \epsilon _{j}}=  a_i ^- (-\tfrac{1}{2}) a_{j} ^+( -\tfrac{1}{2})
 {\1},  \nonumber \\
&& h_{\epsilon _{i}- \epsilon _{j}} = -a_i ^+ (-\tfrac{1}{2}) a_i^-(
-\tfrac{1}{2}) {\1} + a_{j} ^+ (-\tfrac{1}{2}) a_{j}^-(
-\tfrac{1}{2}) {\1}, \nonumber \\
&& \qquad \qquad \qquad \qquad \qquad \qquad \qquad \qquad \qquad
\qquad  \mbox{for }
  i,j =1, \dots, \ell, i<j, \label{generatoriA}
\eea
is level $-1$ affine vertex operator algebra associated to the
affine Lie algebra ${\hat \g}_ {A_{ \ell -1}}$ of type $A_{ \ell
-1}^{(1)}$. By universal property of the vertex algebra  $N_{A_{ \ell -1} ^{(1)}}(-\Lambda_0)$ (cf. \cite{K2}),  we get  a homomorphism of vertex algebras
$$\Phi :  N_{A_{ \ell -1} ^{(1)}}(-\Lambda_0)
 \rightarrow M_{\ell} ^0. $$
 The image of this homomorphism we denote by
 $\widetilde{L}_{A_{ \ell -1} ^{(1)}}
(-\Lambda_0)$, i.e.
\bea \widetilde{L}_{A_{ \ell -1} ^{(1)}} (-\Lambda_0) =
\mbox{Im}(\Phi). \label{L-tilda} \eea

Claim of the following proposition was given in the proof of
Proposition 5 from \cite{AP2} (see also Proposition 2 from that
paper) in the case when $\ell$ is even. The proof in the case when
$\ell$ is odd is analogous.

\begin{proposition} \cite{AP2} \label{lem-konf-vektor} Vertex subalgebra $\widetilde{L}_{A_{ \ell -1} ^{(1)}} (-\Lambda_0)
\otimes M(1)$ of $M_{\ell}$ has the same conformal vector as
$M_{\ell}$. Thus
$$ \omega = \omega_{Sug} + \omega_1, $$
where $\omega_{Sug}$ is conformal vector in $\widetilde{L}_{A_{ \ell -1} ^{(1)}} (-\Lambda_0)$ obtained by the Sugawara construction and
$$\omega_1 = -\frac{1}{2 \ell} H(-1) ^2 {\1}$$
is a conformal vector in $M(1)$.
\end{proposition}
\begin{proof}
We will briefly sketch the proof, and omit details because of its
similarity with the proof of Proposition 2 from \cite{AP2}. Let us
denote
$$ H^{(i)} =  - \sum_{r = 1 }^{i} a_r ^+ (-\tfrac{1}{2}) a_r^-( -\tfrac{1}{2}){\1}
+i a_{i+1} ^+ (-\tfrac{1}{2}) a_{i+1}^-( -\tfrac{1}{2}) {\1}, \quad
\mbox{for }
  i =1, \dots, \ell -1. $$
Using explicit formula for Sugawara conformal vector (cf. \cite{FZ},
\cite{L}, \cite{FrB}, \cite{K2}, \cite{LL}), we obtain:
\begin{eqnarray} && \omega_{Sug} = \frac{1}{2 (\ell -1)} \Big(
\sum_{i,j=1 \atop i<j}^{\ell} (e_{\epsilon _{i}- \epsilon
_{j}}(-1)f_{\epsilon _{i}- \epsilon _{j}}(-1)+ f_{\epsilon _{i}-
\epsilon _{j}}(-1)e_{\epsilon _{i}- \epsilon _{j}}(-1)){\1} \nonumber \\
&& \qquad \quad +\sum_{i=1}^{\ell -1} \frac{1}{i(i+1)} H^{(i)}(-1)
^2 {\1} \Big). \nonumber
 \end{eqnarray}
Using relations (\ref{generatoriA}) we obtain
\bea && (e_{\epsilon _{i} - \epsilon _{j}}(-1)f_{\epsilon _{i} -
\epsilon _{j}}(-1)+f_{\epsilon _{i} - \epsilon _{j}}(-1)e_{\epsilon _{i} - \epsilon _{j}}(-1)){\1} \nonumber \\
&& = 2 a_i ^+ (-\tfrac{1}{2})a_i ^- (-\tfrac{1}{2}) a_j ^+ (-\tfrac{1}{2}) a_j ^- (-\tfrac{1}{2}){\1} \nonumber \\
&& + a_i ^- (-\tfrac{3}{2})a_i ^+ (-\tfrac{1}{2}){\1}  + a_j ^- (-\tfrac{3}{2})a_j ^+ (-\tfrac{1}{2}){\1} \nonumber \\
&& -a_i ^+ (-\tfrac{3}{2})a_i ^- (-\tfrac{1}{2}){\1}  - a_j ^+
(-\tfrac{3}{2})a_j ^- (-\tfrac{1}{2}){\1} \quad (i<j). \nonumber
\eea
Now, direct calculations give:
\bea \omega_{Sug}= \frac{1}{2} \sum_{i=1} ^{\ell} ( a_i ^- (-
\frac{3}{2}) a_i ^{+} (-\frac{1}{2} ) -a_i ^+ (- \frac{3}{2}) a_i
^{-} (-\frac{1}{2} ) ) {\1} + \frac{1}{2 \ell} H(-1) ^{2} {\1},
\nonumber \eea which implies the claim of proposition.
\end{proof}
Proposition \ref{lem-konf-vektor} implies that $\widetilde{L}_{A_{ \ell
-1} ^{(1)}} (-\Lambda_0) \otimes M(1)$--submodule
$$\widetilde{L}_{A_{ \ell -1} ^{(1)}} (-\Lambda_0 + \mu) \otimes
M(1,s)$$
of $M_{\ell}$ has lowest conformal weight
\begin{equation} \label{rel-lowest-confw}
\frac{(\mu , \mu + 2 \rho)}{2(\ell -1)}- \frac{s^2}{2 \ell}.
\end{equation}

Modules $M_{\ell} ^{(s)}$ admit the following $\Zp$--gradation:
$$  M_{\ell} ^{(s)} = \bigoplus_{n = 0 } ^{\infty} M_{\ell} ^{(s)} (n), \qquad M_{\ell} ^{(s)} (n) =
\{ v \in  M_{\ell} ^{(s)} \ \vert \ L(0) v = ( \vert s \vert /2 + n)
v \}. $$ Lowest component of  $M_{\ell} ^{(s)}$ consists of vectors of weight $\vert s \vert /2 $ with respect to $L(0)$ and it is  generated by $a_1 ^+ (-1/ 2) ^s {\1}$
if $s \ge 0$ and by $a_{\ell} ^- (-1/2) ^{-s} {\1}$ if $s <0$.

\begin{lemma} \label{lowest-1} We have:
$$ \widetilde{L}_{A_{ \ell -1} ^{(1)}} (-\Lambda_0) \otimes M(1)  . \ a_1 ^+ (-1/ 2) ^s {\1} \cong
\widetilde{L}_{A_{ \ell -1} ^{(1)}} (- (s+1) \Lambda_0 + s
\Lambda_1) \otimes M(1, s), $$
$$ \widetilde{L}_{A_{ \ell -1} ^{(1)}} (-\Lambda_0) \otimes M(1) . \ a_{\ell}  ^- (-1/ 2) ^s {\1} \cong
\widetilde{L}_{A_{ \ell -1} ^{(1)}} (- (s+1) \Lambda_0 + s
\Lambda_{\ell -1} ) \otimes M(1, -s), $$ where $\widetilde{L}_{A_{
\ell -1} ^{(1)}} (- (s+1) \Lambda_0 + s \Lambda_1)$  and  $
\widetilde{L}_{A_{ \ell -1} ^{(1)}} (- (s+1) \Lambda_0 + s
\Lambda_{\ell -1} ) $ are certain highest weight $A_{ \ell -1}
^{(1)}$--modules with highest weights
$$\lambda_s =- (s+1) \Lambda_0 + s \Lambda_1 \quad \mbox{and} \quad \mu_s = - (s+1) \Lambda_0 + s \Lambda_{\ell -1}.$$
Moreover, lowest  components of these modules are finite-dimensional
$\g _{A_{ \ell -1}}$--modules isomorphic to $V_{A_{ \ell
-1}}(s\omega_1)$ and $V_{A_{ \ell -1}}(s \omega_{\ell -1})$,
respectively.
 \end{lemma}
\begin{proof}
The proof follows from the fact that $a_1 ^+ (-1/ 2) ^s {\1}$ (resp. $a_{\ell}  ^- (-1/ 2) ^s {\1}$)  is a singular vector for ${\hat \g}_ {A_{ \ell -1}}$ of weight $ -(s+1) \Lambda_0 + s
\Lambda_1$ (resp.  $(s+1) \Lambda_0 + s
\Lambda_{\ell -1}$). Since $\g _{A_{ \ell -1}}$--submodule generated by this singular vector belongs to a lowest component of  $M_{\ell} ^{(s)}$, which is finite--dimensional, we get that this module   must be irreducible.
\end{proof}

\section{Main results and steps of the proof}

\label{sect-main}

The main result of the paper is that, for $\ell \ge 3$, every charge
component $M_{\ell} ^{(s)}$ ($s \in \Z$) is an irreducible
$\widehat{\mathfrak{gl}(\ell) }$--module. In this section we list
the main steps of the proof of that claim.

The following proposition is standard (see, for example, Proposition
5.1 of \cite{Ad-jpaa-2005} for a similar version). For completeness we will also include this  proof.

\begin{proposition} \label{ired-1}
$M_{\ell} ^{(0)}$ is a simple vertex subalgebra of $M_{\ell}$. Each
$M_{\ell} ^{(s)}$ ($s \in \Z$) is a simple $M_{\ell}
^{(0)}$--module.
\end{proposition}

\begin{proof}
   Clearly, the charge zero component  $ M_{\ell} ^{(0)}$ is a vertex subalgebra of $M$, and for every $ r , s \in {\Z}$ we have that
  $$ M_{\ell} ^{(r)} \cdot M_{\ell} ^{(s)} = \mbox{span}_{\Bbb C} \{ a_j b \ \vert  \ a \in  M_{\ell} ^{(r)}, \ b \in  M_{\ell} ^{(s)}, \ j \in { \Z} \} \subset  M_{\ell} ^{(r+s)}. $$

  Let $v \in  M_{\ell} ^{(r)}$. Since $M_{\ell}$ is a simple vertex algebra, by Corollary 4.2 of \cite{DM-galois} we have that
  $$ M_{\ell}   = M_{\ell}  . v = \mbox{span}_{\Bbb C} \{ a_j v \ \vert \ a \in M_{\ell}, \ j \in {\Z} \}. $$
  This implies that $ M_{\ell} ^{(r)} =  M_{\ell} ^{(0)}. v $, which  proves simplicity of $ M_{\ell} ^{(r)}$ as $ M_{\ell} ^{(0)}$--module and simplicity of the vertex subalgebra $ M_{\ell} ^{(0)}$.
\end{proof}
The crucial step in proof of the main result is the following
theorem, which is also significant by itself:
\begin{theorem} \label{prosta-0} Assume that $\ell \ge 3$. We have:
 $$M_{\ell} ^{(0)} \cong \widetilde{L}_{A_{ \ell -1} ^{(1)}} (-\Lambda_0) \otimes M(1).$$
 In particular, $\widetilde{L}_{A_{ \ell -1}
^{(1)}} (-\Lambda_0) = U(\hg _{A_{ \ell -1}}). {\1}$ is a simple
vertex operator algebra.
 \end{theorem}

The proof of Theorem \ref{prosta-0} will use the classification
result from \cite{AP}, which we recall in Section \ref{sect-aff-1},
and certain fusion rules arguments. We present that proof in
separate Section \ref{sect-irr-charge}.

\begin{theorem} \label{irr-charge} Assume that $\ell \ge 3$ and $s \in \Z$. Then $M_{\ell} ^{(s)}$ is an
irreducible $\widehat{\mathfrak{gl}(\ell) }$--module. In particular
\bea
M_{\ell} ^{(s)} \cong  L_{A_{ \ell -1} ^{(1)}} (- (s+1) \Lambda_0 + s \Lambda_{1} ) \otimes M(1, s) \qquad (s \ge 0),  \nonumber \\
M_{\ell} ^{(s)} \cong  L_{A_{ \ell -1} ^{(1)}} ( (s-1) \Lambda_0 - s
\Lambda_{\ell -1} ) \otimes M(1, s) \qquad (s <  0).  \nonumber \eea

\end{theorem}
 \begin{proof}
 Theorem \ref{prosta-0} and Proposition \ref{ired-1} give that $M_{\ell} ^{(s)}$ is an irreducible
 $L_{ A_{\ell -1} ^{(1)} }(-\Lambda_0) \otimes M(1)$--module. Now assertion follows from Lemma \ref{lowest-1}.
 \end{proof}

\section{Affine vertex algebras of type $A$ and level $-1$ }

\label{sect-aff-1} Our proof of main result (Theorem \ref{prosta-0})
will use the representation theory of the vertex algebra
$\widetilde{L}_{A_{ \ell -1} ^{(1)}} (-\Lambda_0)$ realized as a
subalgebra of the Weyl vertex algebra $M_{\ell}$. In this section we
shall see that $\widetilde{L}_{A_{ \ell -1} ^{(1)}} (-\Lambda_0)$ is
a certain quotient of the vertex algebra $V_{\ell -1,1}$ studied in
\cite{AP}. Then the classification  result for $V_{\ell-1,1}$ will
automatically give the classification of irreducible
$\widetilde{L}_{A_{ \ell -1} ^{(1)}} (-\Lambda_0)$--modules.  (It is
also possible that  $\widetilde{L}_{A_{ \ell -1} ^{(1)}}
(-\Lambda_0) \cong V_{\ell -1,1}$).

So we shall first  recall some of our previous results from
\cite{AP} on affine vertex algebras of type $A$ and level $-1$, and
apply them to the present situation.

In \cite{AP}, we studied the following quotient of universal affine
vertex algebra $N_{ A_{\ell -1} ^{ (1)} } (- \Lambda_0)$, denoted by
$V_{\ell -1,1}$ in that paper:
$$V_{\ell -1,1} = \frac{N_{ A_{\ell -1} ^{ (1)} } (- \Lambda_0)}{< v_{\ell -1,1} >}, $$
where $< v_{\ell -1,1} >$ denotes the ideal generated by singular
vector
$$v_{\ell -1,1} = \left\{ \begin{array}{cc}
                                          e_{\epsilon_1 - \epsilon_{\ell}}(-1)e_{\epsilon_2 -
\epsilon_{\ell-1}}(-1) {\bf 1} - e_{\epsilon_2 - \epsilon_{\ell}}(-1)e_{\epsilon_1 - \epsilon_{\ell -1}}(-1) {\bf 1}, &   \ \ell \ge 4 \\
                                           e_{\epsilon_2 - \epsilon_3}(-1)^2 e_{\epsilon_1 -
\epsilon_2}(-1) {\bf 1} - e_{\epsilon_1 - \epsilon_3}(-1)
e_{\epsilon_2 - \epsilon_3}(-1)h_{\epsilon_1 -
\epsilon_2}(-1){\bf 1} \\
 - e_{\epsilon_1 -
\epsilon_3}(-1)^2 f_{\epsilon_1 - \epsilon_2}(-1){\bf 1}, & \ \ell
=3
                                        \end{array}  \right .$$
in $N_{ A_{\ell -1} ^{ (1)} } (- \Lambda_0)$.

Using the method for classification of representations of affine
vertex algebras based on singular vectors (from \cite{Ad-1994},
\cite{AM}, \cite{MP}), we obtained the following result (see Theorem
5.7 and Corollary 6.7 from \cite{AP}):
\begin{proposition} \cite{AP} \label{prop-class-JA-V}
Let $\ell \geq 3$.  The set
$$ \{ L_{ A_{\ell-1} ^{(1)} }(-(s+1) \Lambda_0 + s \Lambda_1), \ L_{ A_{\ell-1} ^{(1)} }(-(s+1) \Lambda_0 + s \Lambda_{\ell -1}) \ \vert \ s \in {\Zp}\} $$
provides a complete list of irreducible ordinary $V_{\ell
-1,1}$--modules.
\end{proposition}

 The following result is a  modification of Proposition \ref{prop-class-JA-V} which we  need in this paper.
\begin{proposition} \label{prop-class-JA}
Let $\ell \geq 3$.

\item[(i)] There is a surjective homomorphism of vertex algebras
$$ \widetilde{\Phi} : V_{\ell-1,1} \rightarrow \widetilde{L}_{A_{ \ell -1} ^{(1)}}
(-\Lambda_0).$$

\item[(ii)]  The set
$$ \{ L_{ A_{\ell-1} ^{(1)} }(-(s+1) \Lambda_0 + s \Lambda_1), \ L_{ A_{\ell-1} ^{(1)} }(-(s+1) \Lambda_0 + s \Lambda_{\ell -1}) \ \vert \ s \in {\Zp}\} $$
provides a complete list of irreducible ordinary $\widetilde{L}_{A_{
\ell -1} ^{(1)}} (-\Lambda_0)$--modules.
\end{proposition}
\begin{proof}

Using homomorphism $\Phi : N_{A_{ \ell -1} ^{(1)}}
(-\Lambda_0)\rightarrow \widetilde{L}_{A_{ \ell -1} ^{(1)}}
(-\Lambda_0) $ defined in Section \ref{sect-weyl-vertex} and
formulas for  singular vectors $ v_{\ell -1,1} $ one can easily
verify that $\Phi ( v_{\ell -1,1} ) = 0$  in $M_{\ell}$, for $\ell
\geq 3$. This implies that $$\mbox{Ker}(\Phi) \supset < v_{\ell
-1,1}
>, $$ and therefore we get homomorphism $\widetilde{\Phi} :
V_{\ell-1,1} \rightarrow \widetilde{L}_{A_{ \ell -1} ^{(1)}}
(-\Lambda_0)$.

 Lemma \ref{lowest-1} now implies that    $$  \widetilde{ L}_{ A_{\ell-1} ^{(1)} }(-(s+1) \Lambda_0 + s \Lambda_1), \ \widetilde{L}_{ A_{\ell-1} ^{(1)} }(-(s+1) \Lambda_0 + s \Lambda_{\ell -1}) \quad (s \in {\Zp}) $$
are $\widetilde{L}_{A_{ \ell -1} ^{(1)}} (-\Lambda_0)$--modules, and
therefore their simple quotients are also $\widetilde{L}_{A_{ \ell
-1} ^{(1)}} (-\Lambda_0)$--modules. Now assertion follows from
Proposition \ref{prop-class-JA-V} since every $\widetilde{L}_{A_{
\ell -1} ^{(1)}} (-\Lambda_0)$--module is also a $V_{\ell
-1,1}$--module.
\end{proof}

\section{Proof of Theorem \ref{prosta-0}}

\label{sect-irr-charge}

In this section we give the proof of Theorem \ref{prosta-0}.

The following lemma is a crucial technical result needed for the
fusion rules argument used in the proof of Theorem \ref{prosta-0}.

\begin{lemma} \label{vazna-1}
Let $r, s \in {\N}$, $r \geq s$ and $\ell \ge 2$. Then the following
tensor product decompositions hold:
\bea &(i)& V_{A_{\ell} } ( r \omega_1) \otimes V_{A_{\ell} } (s
\omega_1) = \nonumber \\ && V_{A_{\ell} } ( (r +s ) \omega_1) \oplus
V_{A_{\ell} } ( (r +s-2  ) \omega_1 + \omega_2)
\oplus \cdots \oplus V_{A_{\ell} } ( (r -s ) \omega_1 + s \omega_2), \nonumber \\
&(ii)& V_{A_{\ell} } ( r \omega_{\ell}) \otimes V_{A_{\ell} } (s
\omega_{\ell}) = \nonumber \\ && V_{A_{\ell} } ( (r +s )
\omega_{\ell} ) \oplus V_{A_{\ell} } ( (r +s-2  ) \omega_{\ell} +
\omega_{\ell -1})
\oplus \cdots \oplus V_{A_{\ell} } ( (r -s ) \omega_{\ell} + s \omega_{\ell-1}), \nonumber \\
 &(iii)& V_{A_{\ell} } ( r \omega_1) \otimes V_{A_{\ell} } (s \omega_{\ell}) = \nonumber \\ && V_{A_{\ell} } ( (r -s ) \omega_1) \oplus
 V_{A_{\ell} } ( (r -s+1  ) \omega_1 + \omega_{\ell})
\oplus \cdots \oplus V_{A_{\ell} } ( r \omega_1 + s \omega_{\ell}).
\nonumber
\eea

\end{lemma}

Note that all weight subspaces of modules $V_{A_{\ell} } ( r
\omega_{1})$ and  $V_{A_{\ell} } ( r \omega_{\ell})$  are
$1$--dimensional, and decompositions can be obtained by using
Kostant formula \cite{Ko}. Decompositions of this type also appeared
in \cite{Okada}.

\vskip 5mm

\noindent {\bf Proof of Theorem \ref{prosta-0}:}

\vskip 3mm

First we notice that   $M_{\ell} ^{(0)}$ is completely reducible
$M(1)$--module, isomorphic to an infinite direct sum of copies of
$M(1)$.

Assume now that $ \widetilde{L}_{A_{ \ell -1} ^{(1)}} (-\Lambda_0)
\otimes M(1) $ is a proper submodule of $M_{\ell} ^{(0)}$. Then we
consider $M_{\ell} ^{(0)}$ as a $ \widetilde{L}_{A_{ \ell -1}
^{(1)}} (-\Lambda_0) \otimes M(1)$--module and conclude that there
exists a singular vector $\Omega$  which is not proportional to
${\bf 1}$. Applying the classification result from Proposition
\ref{prop-class-JA}, we may assume that  $\Omega$ has weight
$\lambda_r$ or   $\mu_r$ for certain $r \in {\N}$.

Assume first that $\Omega$ has weight $\lambda_r$. Then $U(\g _{A_{
\ell -1}}). \Omega \cong V_{A_{\ell -1}}(r \omega_1)$.

Next we consider $M_{\ell} ^{(0)}$--module $M_{\ell} ^{(- r)}$ which
is generated by vector $ a_{\ell} ^{-} (-1/ 2) ^{ r} {\bf 1}$.

Since $\Omega$ and $ a_{\ell} ^{-} (-1/ 2) ^{ r} {\bf 1}$ belong to
the simple vertex algebra $M_{\ell}$, we conclude that

$$Y(\Omega,z)  a_{\ell} ^{-} (-1/ 2) ^{ r} {\bf 1} \ne 0. $$
Take $n_0$ such that
$$ \Omega_{n_0}  a_{\ell} ^{-} (-1/ 2) ^{ r} {\bf 1} \ne 0, \quad \Omega_n   a_{\ell} ^{-} (-1/ 2) ^{ r} {\bf 1} = 0 \ \mbox{for} \ n > n_0. $$
Then $$U = U(\g _{A_{ \ell -1}}). \Omega_{n_0}  a_{\ell} ^{-} (-1/
2) ^{ r} {\bf 1}$$ is a $U(\g _{A_{ \ell -1}})$--module which is the
lowest component of $\widetilde{L}_{A_{ \ell -1} ^{(1)}}
(-\Lambda_0)$--module $\widetilde{U} = U(\hg _{A_{ \ell -1}}).
\Omega_{n_0} a_{\ell} ^{-} (-1/ 2) ^{ r} {\bf 1}$. Then Lemma
\ref{vazna-1} (iii) and the classification result from Proposition
\ref{prop-class-JA} give that $U$ is $1$--dimensional. Relation
(\ref{rel-lowest-confw}) now implies that $M_{\ell} ^{(-r)}$
contains a non-trivial singular vector of conformal weight
$-\frac{r^2}{2 \ell}$. This is a contradiction since $M_{\ell}
^{(-r)}$ does not have vectors at this conformal weight (lowest
conformal weight is $r/2$).

Similarly we treat the case if $\Omega$ has weight $\mu_r$.

In this way we have proved that $M_{\ell} ^{(0)} \cong \widetilde{L}_{A_{ \ell -1} ^{(1)}} (-\Lambda_0) \otimes M(1).$

\section{Some applications}
\label{sect-semisimple}

In this section we shall present some direct applications of the
results from previous sections. We shall complete the classification
process for simple affine vertex operator algebras $L_{ A_{\ell -1}
^{(1)} }(-\Lambda_0)$ and $L_{ C_{\ell} ^{(1)}} (-\Lambda_0)$
started in papers \cite{A2}, \cite{AP} and \cite{AP2}.

\subsection{Fusion rules for $L_{ A_{\ell -1} ^{(1)} }(-\Lambda_0)$--modules }

As a consequence of our construction and Proposition
\ref{prop-class-JA}, we get the following classification result.

\begin{theorem}
Let $\ell \geq 3$.  The set
$$ \{ L_{ A_{\ell-1} ^{(1)} }(-(s+1) \Lambda_0 + s \Lambda_1), \ L_{ A_{\ell-1} ^{(1)} }(-(s+1) \Lambda_0 + s \Lambda_{\ell -1}) \ \vert \ s \in {\Zp}\} $$
provides a complete list of irreducible ordinary $L_{ A_{\ell -1} ^{(1)} }(-\Lambda_0)$--modules.

\end{theorem}

We are also able to describe the fusion rules between irreducible
$L_{ A_{\ell -1} ^{(1)} }(-\Lambda_0)$--modules, for $\ell \geq 3$.

For every $s \in \Z$ we set $\pi_s = \left\{ \begin{array}{cc}
                                          L_{ A_{\ell -1} ^{(1)} }(-(s+1) \Lambda_0 + s \Lambda_1) &  \mbox{if} \ s \ge 0 \\
                                           L_{ A_{\ell -1} ^{(1)} } ((s-1) \Lambda_0 -  s \Lambda_{\ell -1} ) & \mbox{if} \ s < 0
                                        \end{array}  \right .$

Next theorem will show that the fusion algebra generated by
irreducible (ordinary) $L_{ A_{\ell -1}^{(1)}}(-\Lambda_0)$--modules
is isomorphic to the group algebra ${\C}[\Z]$.
\begin{theorem}
Assume that $i, j, k \in \Z$. The there exists a non-trivial intertwining operator of type
$$ { \pi_k \choose \pi_i \ \ \pi_j } $$
if and only if $k = i+j$
\end{theorem}
\begin{proof}
Let $Y_{M_{\ell}}(\cdot,z)$ be the vertex operator map which defines on $M_{\ell}$ the structure of vertex operator algebra.
Required  non-trivial intertwining operators are realized by restriction, as follows:
$$ I(u,z ) v = Y_{M_{\ell}} (u, z) v \qquad (u \in M_{\ell} ^{(i)}, v \in M_{\ell} ^{(j)} ). $$
The rest of statement easily follows from Lemma \ref{vazna-1} by using standard fusion rules arguments.
\end{proof}

\begin{remark} Vertex operator algebra $L_{ A_{\ell
-1}^{(1)}}(-\Lambda_0)$ is not rational in the category ${\mathcal
O}$, unlike admissible affine vertex operator algebras (cf. \cite{Ad-1994}, \cite{AM},
\cite{Ara}, \cite{P3}, \cite{P4}). Modules constructed in Remark 5.8 of \cite{AP} are
examples of $L_{ A_{\ell -1}^{(1)}}(-\Lambda_0)$--modules from
category ${\mathcal O}$ that are not completely reducible.
\end{remark}

\subsection{  $ M_{2 \ell} $ as a  $C_{\ell} ^{(1)}$--module}
Next consequence of the construction from Section \ref{sect-irr-charge} is the complete reducibility of  $ M_{2 \ell} $ as a module for
affine Lie algebra $C_{\ell} ^{(1)}$.

By using conformal embedding  of $L_{ C_{\ell} ^{(1)}} (-\Lambda_0)$
into  ${L}_{A_{ 2 \ell -1} ^{(1)}}(-\Lambda_0)$, the classification
results from \cite{A2}, \cite{AP2}  and Proposition 7 from \cite{AP2} we get:

\begin{corollary}

\item[(i)] Let $\ell \geq 3$. The set
$$ \{ L_{ C_{\ell} ^{(1)} } (-(s+1) \Lambda_0 + s \Lambda_1 ) \ \vert \ s \in {\Zp} \} \cup \{ L_{C_{\ell} ^{(1)} } (-2 \Lambda_0 + \Lambda_2 ) \} $$
 provides a complete list of irreducible ordinary $L_{ C_{\ell} ^{(1)}
 }(-\Lambda_0)$--modules.

\item[(ii)] $M_{2 \ell}$ is a completely reducible $L_{ C_{\ell} ^{(1)}} (-\Lambda_0)$--module.

\end{corollary}

\subsection{  On $\mathcal{W}_{1 + \infty}$-algebra}

 Let $\mathcal{W}_{1 + \infty, -\ell}$ denote the simple $\mathcal{W}_{1 + \infty}$ vertex algebra with central charge $c=-\ell$. Then (cf. \cite{KR})

 $$\mathcal{W}_{1 + \infty, -\ell} = M_{\ell} ^{ \mathfrak{gl}(\ell) }.$$

 Our results show that

$$\mathcal{W}_{1 + \infty, -\ell} = \left( L_{ A_{\ell -1} ^{(1)} }(-\Lambda_0) \otimes M(1) \right) ^{ \mathfrak{gl}(\ell) }
\cong L_{ A_{\ell -1} ^{(1)} }(-\Lambda_0)  ^{ \mathfrak{sl}(\ell) } \otimes M(1) $$
if $\ell \ge 3$. (For $\ell \le 2$ this is not true).

 Therefore, $ \mathcal{W}_{1 + \infty, -\ell} $ is an affine orbifold.

\begin{remark} As far as we understand, the above result is new in the mathematical literature.
Note also that this result is in complete agreement with result from
\cite{FKRW}:
 $$\mathcal{W}_{1 + \infty, \ell} = \left( L_{ A_{\ell -1} ^{(1)} }(\Lambda_0) \otimes M(1) \right) ^{ \mathfrak{gl}(\ell) }
 \cong L_{ A_{\ell -1} ^{(1)} }(\Lambda_0)  ^{ \mathfrak{sl}(\ell) } \otimes M(1). $$
\end{remark}

\section{  Fusion algebra for certain  $D_{2 \ell-1}^{(1)}$--modules  of negative levels}
\label{sect-fusionD}

 Let $N_{ D_{\ell} ^{ (1)} } (k \Lambda_0)$ be the universal affine vertex algebra of level
$k$ associated to the affine Lie algebra of type $D_{\ell} ^{(1)}$.
It was proved in \cite{P2} that the vector
\begin{equation} \label{rel-singvD-lowest}
v = \sum _{i=2}^{\ell} e_{\epsilon_1 - \epsilon_i}(-1) e_{\epsilon_1
+ \epsilon_i}(-1) {\bf 1}
\end{equation}
is a singular vector in $N_{D_{\ell} ^{(1)}}((-\ell+2)\Lambda_0)$.
The classification of representations of vertex operator algebra
$${\mathcal V}_{D_{\ell} ^{(1)}}((-\ell+2)\Lambda_0)= \frac{N_{D_{\ell} ^{(1)}}((-\ell+2)\Lambda_0)}
{< v > },$$
where $< v >$ is the ideal generated by singular vector $v$, was
studied in that paper. We will recall the classification of all
irreducible ordinary modules for ${\mathcal V}_{D_{\ell}
^{(1)}}((-\ell+2)\Lambda_0)$.

It follows from \cite{P2} that:
\begin{proposition} \label{class-d}
The set
$$ \{ L_{ D_{ \ell} ^{(1)} } (-(s+\ell -2) \Lambda _0 + s\Lambda_{\ell -1} ), L _{ D_{\ell} ^{(1)} } (-(s+\ell -2) \Lambda _0 + s\Lambda_{\ell} )
\ \vert \ s \in {\Zp} \}$$
gives all irreducible ordinary ${\mathcal V}_{D_{\ell}
^{(1)}}((-\ell+2)\Lambda_0)$--modules.
\end{proposition}

In this section we shall prove that the above modules actually give
all irreducible modules for the simple vertex operator algebra
$L_{D_{\ell} ^{(1)}}((-\ell+2)\Lambda_0)$ in the case when $\ell$ is
odd.

Top components of irreducible modules from Proposition \ref{class-d}
are irreducible modules for the simple Lie algebra of type $D$. So
we get an interesting series of modules:
 $$ V_{D_{\ell} } (s \omega_{\ell -1}), \  V_{D_{\ell}} (s \omega_{\ell})   \qquad (s \in {\Zp}). $$

 Let $$U_{D_{\ell} }  (s)  :=\left\{ \begin{array}{cc}
                               V_{ D_{\ell} } (s \omega_{\ell -1} ),  &  \mbox{for} \ s \ge 0  \\
                               V_{ D_{\ell} } (-s  \omega_{\ell } ), &  \mbox{for} \ s < 0.
                             \end{array} \right .
   $$
 These modules appear in the representation theory of simple Lie algebras.
 Particulary important result is obtained by S. Okada \cite{Okada}.
 He described characters of these representations and  their tensor products.
 He proved that their tensor product is multiplicity free, i.e., every irreducible component has multiplicity one.

 By using Theorem 2.5 of \cite{Okada} we have:
 \begin{proposition} \label{tensor-D-Okada} Assume that $\ell \ge 3$ is an odd natural number.
 Assume that $r, s \in {\Z}$.  Then $U_{D_{\ell} }  (t)$ appears in the tensor product $U_{D_{\ell} }  (r) \otimes U_{D_{\ell} }  (s)$
 if and only if $t = r +s$. The multiplicity is one.
 \end{proposition}
\begin{proof}
We present the proof in the case $r \geq 0$ and $s \geq 0$. Other
cases are treated similarly. The combinatorial description of tensor
product decomposition from Theorem 2.5 (3) (for $\ell$ odd) of
\cite{Okada} gives that
$$ V_{ D_{\ell} } (r \omega_{\ell -1} ) \otimes V_{ D_{\ell} } (s
\omega_{\ell -1} ) = \bigoplus _{\mu} V_{D_{\ell} }( \mu),$$
where the summation goes over dominant integral weights $\mu$ such
that
$$\mu = (k_1 \omega _1 + k_3 \omega _3 + k_5 \omega
_5 + \ldots + k_{\ell -2} \omega _{\ell -2}) + k_{\ell -1} \omega
_{\ell -1} \quad (k_i \in \Zp),$$
and
$$ 2(k_1 + k_3 + \ldots + k_{\ell -2}) + k_{\ell -1}= r+s.$$
This implies that modules $V_{ D_{\ell} } (t \omega_{\ell} )$ do not
appear in that tensor product and that $V_{ D_{\ell} } (t
\omega_{\ell -1} )$ appears only for $t = r +s$ (with multiplicity
one). The claim follows.
\end{proof}

This tensor product decomposition implies the following result about
fusion rules:
 \begin{theorem} \label{thm-int-d}
 Assume that $\widetilde{\pi_r}$, $r \in \Z$  are $\Zp$--graded ${\mathcal V}_{D_{2\ell-1} ^{(1)}}((-2\ell+3)\Lambda_0)$-modules such that top
 component of $\widetilde{\pi_r}$ is isomorphic to the irreducible $\frak{g}_{D_{2\ell-1} }$--module $U_{D_{2\ell-1} }  (r)$.
 Let $\pi_r$ denote the associated simple quotient. Assume that there is a non-trivial intertwining operator of type
 $${ \pi_t \choose \ \widetilde{\pi_r} \ \ \pi_s }.$$
 Then $t = r+s$.
 \end{theorem}

\begin{remark} It is interesting that in the case of  ${\mathcal V}_{D_{2\ell} ^{(1)}}((-2\ell+2)\Lambda_0)$--modules
we get a different fusion algebra.
\end{remark}

Now we shall apply our result on fusion rules to get complete
classification of irreducible ordinary $L_{D_{2 \ell -1} ^{(1)}}
(-(2 \ell -3) \Lambda_0)$--modules.

\begin{theorem} \label{thm-class-D-neparne}
 The set
\bea && \left\{ L_{D_{2 \ell -1}  ^{(1)}}(-(s+ 2 \ell -3)\Lambda_0+
s \Lambda _{2\ell-2}), L_{D_{2 \ell -1} ^{(1)}}(-(s+2 \ell
-3)\Lambda_0+ s \Lambda _{2\ell-1}) \ \vert \ s \in \Z _{\geq 0}
\right \} \nonumber \\
\label{set-1}
\eea
provides the complete list of irreducible ordinary $L_{D_{ 2 \ell -1} ^{(1)}}
(-( 2 \ell -3) \Lambda_0)$--modules.
\end{theorem}
\begin{proof} From Proposition \ref{class-d} we get  that every irreducible module
for the simple vertex operator algebra $L_{D_{ 2 \ell -1} ^{(1)}}
(-( 2 \ell -3) \Lambda_0)$ belongs to the set (\ref{set-1}). It
remains to prove that every module $(M, Y_M(\cdot,z) )$ from the set
(\ref{set-1}) is in fact a module for $L_{D_{ 2 \ell -1} ^{(1)}} (-(
2 \ell -3) \Lambda_0)$.

Let $J (-(2 \ell-3) \Lambda_0)$ be the maximal ideal in the vertex
operator algebra ${\mathcal V} _{D_{ 2 \ell -1} ^{(1)}}(-(2 \ell-3)
\Lambda_0)$. Since $J (-(2 \ell-3) \Lambda_0)$ is an ordinary
${\mathcal V} _{D_{ 2 \ell -1} ^{(1)}} (-(2 \ell-3)
\Lambda_0)$--module, we conclude that it is generated by singular
vectors which have weights $\lambda_r = -(r+ 2 \ell -3) \Lambda_0 +
r \Lambda_{ 2 \ell -2}$ or $\mu_r =- (r+ 2 \ell -3) \Lambda_0 + r
\Lambda_{2 \ell-1}$, $r \in \N$. Assume that there is a singular
vector  $\Omega$ of weight $\lambda_r$ or $\mu_r$  such that vertex
operator $Y_M(\Omega,z)$ is non-zero on $M$. This leads to a
non-trivial intertwining operator of type
$$ { M \choose \widetilde{L(\lambda_r)} \ \ M  } \quad \mbox{or} \quad { M \choose \widetilde{L(\mu_r)} \ \ M }.$$
This contradicts Theorem \ref{thm-int-d}.

Therefore, vertex operators $Y_M(v,z)$ must vanish for all $v \in
J(-(2 \ell-3) \Lambda_0)$. This proves that $M$ is an $L_{D_{ 2 \ell
-1} ^{(1)}} (-( 2 \ell -3) \Lambda_0)$--module.
\end{proof}

\section{Conformal embedding of $L_{D_{5} ^{(1)}}(-3\Lambda_0) \otimes M(1)$ into $L_{E_{6} ^{(1)}}(-3\Lambda_0)$ }
\label{sect-conf-emb}

In this section we use the methods from Sections \ref{sect-main} and
\ref{sect-irr-charge} and the fusion rules analysis from Section
\ref{sect-fusionD} to construct the conformal embedding of $L_{D_{5}
^{(1)}}(-3 \Lambda_0) \otimes M(1)$ into $L_{E_{6} ^{(1)}}(-3
\Lambda_0)$. We also determine the decomposition of $L_{E_{6}
^{(1)}}(-3 \Lambda_0)$ into direct sum of $L_{D_{5} ^{(1)}}(-3
\Lambda_0) \otimes M(1)$--modules.

Let ${\frak g}_{E_6}$ be the simple Lie algebra of type $E_6$. We
will use the construction of the root system of type $E_6$ from
\cite{Bou}, \cite{H}, and the notation for root vectors from
\cite{AP-ART}. Note now that the subalgebra of ${\frak g}_{E_6}$
generated by (suitably chosen) positive root vectors
\bea && e_{(5)}=e_{ \frac{1}{2}( \epsilon_8 - \epsilon_7 -
\epsilon_6 +
\epsilon_5 - \epsilon_4 - \epsilon_3 - \epsilon_2 - \epsilon_1 )}, \nonumber \\
&&  e_{\alpha _2}= e_{\epsilon _2 + \epsilon _1}, \nonumber \\
&& e_{\alpha _4}= e_{\epsilon _3 - \epsilon _2}, \nonumber \\
&& e_{\alpha _3}= e_{\epsilon _2 - \epsilon _1}, \nonumber \\
&& e_{\alpha _5}= e_{\epsilon _4 - \epsilon _3} \nonumber  \eea
and associated negative root vectors is a simple Lie algebra of type
$D_5$. We denote this Lie algebra by ${\frak g}_{D_5}$. Thus,
${\frak g}_{E_6}$ has a reductive subalgebra ${\frak g}_{D_5} \oplus
{\frak h}$, where ${\frak h}= \C H$, and
$$H= \frac{1}{3}(h_8-h_7-h_6-3h_5)$$
(where $h_i$ are determined by $\epsilon_i(h_j)=\delta_{ij}$), with
the associated decomposition
\bea \label{decompLieDE} {\frak g}_{E_6} = {\frak g}_{D_5} \oplus
{\frak h} \oplus V_{D_5 } (\omega_4) \otimes \C _1
 \oplus V_{D_5 } (\omega_5) \otimes \C _{-1}, \eea
where $\C _r$ denotes the one-dimensional ${\frak h}$--module on
which $H$ acts as scalar $r \in \C$. The highest weight vectors of
${\frak g}_{D_5}$--modules $V_{D_5 } (\omega_4)$ and $V_{D_5 }
(\omega_5)$ are $e_{(234)}$ and $e_{\epsilon _5 + \epsilon _4}$,
respectively.

Denote by $\widetilde{L}_{D_5 ^{(1)}} (-3 \Lambda_0)$ the subalgebra
of $L_{E_{6} ^{(1)}}(-3 \Lambda_0)$ generated by ${\frak g}_{D_5}$
and by $M(1)$ the Heisenberg vertex subalgebra of $L_{E_{6}
^{(1)}}(-3 \Lambda_0)$ generated by~$H$. Relation
(\ref{decompLieDE}) and the criterion for conformal embeddings from
\cite{AP-ART} give:
\begin{proposition}
Vertex operator algebra $\widetilde{L}_{D_5 ^{(1)}} (-3 \Lambda_0)
\otimes M(1)$ is conformally embedded in $L_{E_{6} ^{(1)}}(-3
\Lambda_0)$.
\end{proposition}

As before, denote by $M(1, s)$ the irreducible $M(1)$--module on
which $H(0)$ acts as scalar $s \in {\C}$. It follows that
$\widetilde{L}_{D_5 ^{(1)}} (-3 \Lambda_0) \otimes M(1)$--submodule
$$\widetilde{L}_{D_5 ^{(1)}} (-3 \Lambda_0 + \mu) \otimes
M(1,s)$$
of $L_{E_{6} ^{(1)}}(-3 \Lambda_0)$ has lowest conformal weight
\begin{equation} \label{rel-lowest-confw-D}
\frac{(\mu , \mu + 2 \rho)}{10}- \frac{s^2}{8}.
\end{equation}

It was shown in \cite{AP-ART} that
\begin{eqnarray*}
&& v_{E_{6}} = ( e_{(5)}(-1)e_{(12345)}(-1) +
e_{(125)}(-1)e_{(345)}(-1) \\
&& \qquad + e_{(135)}(-1)e_{(245)}(-1) + e_{(235)}(-1)e_{(145)}(-1)
) {\bf 1}
\end{eqnarray*}
is a singular vector in $N_{E_{6} ^{(1)}}(-3 \Lambda_0)$. Thus, this
vector is trivial in $L_{E_{6} ^{(1)}}(-3 \Lambda_0)$.

The subalgebra of $N_{E_{6} ^{(1)}}(-3 \Lambda_0)$ generated by
${\frak g}_{D_5}$ is isomorphic to $N_{D_{5} ^{(1)}}(-3 \Lambda_0)$.
Note that $v_{E_{6}} \in N_{D_{5} ^{(1)}}(-3 \Lambda_0)$, and that
in fact $v_{E_{6}}$ is equal to vector $v$ from relation
(\ref{rel-singvD-lowest}), for $\ell=5$. Since $v_{E_{6}}$ is
trivial in $L_{E_{6} ^{(1)}}(-3 \Lambda_0)$, we conclude that
$\widetilde{L}_{D_5 ^{(1)}} (-3 \Lambda_0)$ is a quotient of vertex
algebra ${\mathcal V}_{D_5 ^{(1)}}(-3 \Lambda_0)$ from
Section~\ref{sect-fusionD}. The classification results from
Proposition \ref{class-d} and Theorem \ref{thm-class-D-neparne} now
give:
\begin{proposition} \label{prop-class-tilda}
 The set
\begin{eqnarray*}
\left\{ L_{D_5 ^{(1)}}(-(s+3)\Lambda_0+ s \Lambda _4), L_{D_5
^{(1)}}(-(s+3)\Lambda_0+ s \Lambda _5) \ \vert \ s \in \Z _{\geq 0}
\right\}
\end{eqnarray*}
provides the complete list of irreducible ordinary
$\widetilde{L}_{D_5 ^{(1)}} (-3 \Lambda_0)$--modules.
\end{proposition}

Vectors $e_{(234)}(-1)^s {\bf 1}$ and $e_{\epsilon _5 + \epsilon
_4}(-1)^s {\bf 1}$ are (non-trivial) singular vectors for
$\hat{\frak g}_{D_5}$ in $L_{E_{6} ^{(1)}}(-3 \Lambda_0)$ of highest
weights $-(s+3)\Lambda_0+ s \Lambda _4$ and $-(s+3)\Lambda_0+ s
\Lambda _5$, respectively. Furthermore, we have:
\bea && [H(0),e_{(234)} (n)] =  e_{(234)} (n), \nonumber \\
&& [H(0),e_{\epsilon _5 + \epsilon _4} (n)] = - e_{\epsilon _5 +
\epsilon _4} (n). \nonumber \eea

We have:
$$ \widetilde{L}_{D_5 ^{(1)}} (-3\Lambda_0) \otimes M(1)  . e_{(234)}(-1)^s {\bf 1} \cong
\widetilde{L}_{D_5 ^{(1)}} (-(s+3)\Lambda_0+ s \Lambda _4) \otimes
M(1, s), $$
$$ \widetilde{L}_{D_5 ^{(1)}} (-3\Lambda_0) \otimes M(1)  . e_{\epsilon _5 + \epsilon
_4}(-1)^s {\bf 1} \cong \widetilde{L}_{D_5 ^{(1)}} (-(s+3)\Lambda_0+
s \Lambda _5) \otimes M(1, -s), $$
where $\widetilde{L}_{D_5 ^{(1)}} (-(s+3)\Lambda_0+ s \Lambda _4)$
and  $ \widetilde{L}_{D_5 ^{(1)}} (-(s+3)\Lambda_0+ s \Lambda _5)$
are certain highest weight $D_5 ^{(1)}$--modules with corresponding
highest weights.

Clearly, $H(0)$ acts semisimply on $L_{E_{6} ^{(1)}}(-3 \Lambda_0)$
and it defines the following ${\Z}$--gradation:
$$
L_{E_{6} ^{(1)}}(-3 \Lambda_0) = \bigoplus_{ s \in \Z} L_{E_{6}
^{(1)}}(-3 \Lambda_0) ^{(s)} , \quad L_{E_{6} ^{(1)}}(-3 \Lambda_0)
^{(s)} = \{ v \in L_{E_{6} ^{(1)}}(-3 \Lambda_0) \ \vert \ H(0) v =
s v \}.
$$

The following proposition is analogous to Proposition \ref{ired-1}:

\begin{proposition} \label{ired-1-DE}
$L_{E_{6} ^{(1)}}(-3 \Lambda_0) ^{(0)}$ is a simple vertex
subalgebra of $L_{E_{6} ^{(1)}}(-3 \Lambda_0)$. Each $L_{E_{6}
^{(1)}}(-3 \Lambda_0) ^{(s)}$ ($s \in \Z$) is a simple $L_{E_{6}
^{(1)}}(-3 \Lambda_0) ^{(0)}$--module.
\end{proposition}

Clearly, $L_{E_{6} ^{(1)}}(-3 \Lambda_0) ^{(s)}$ is generated as
$L_{E_{6} ^{(1)}}(-3 \Lambda_0) ^{(0)}$--module by $e_{(234)}(-1)^s
{\bf 1}$ for $s \in \Z _{\geq 0}$, and by $e_{\epsilon _5 + \epsilon
_4}(-1)^{-s} {\bf 1}$ for $s \in \Z _{< 0}$.

\begin{theorem} \label{prosta-0-DE}
\item[(a)]
We have:
 $$L_{E_{6} ^{(1)}}(-3\Lambda_0) ^{(0)} \cong \widetilde{L}_{D_5 ^{(1)}} (-3\Lambda_0) \otimes M(1).$$
In particular, $\widetilde{L}_{D_5 ^{(1)}} (-3\Lambda_0)$ is a
simple vertex operator algebra.

\item[(b)]  Every $L_{E_{6} ^{(1)}}(-3\Lambda_0) ^{(s)}$ ($s \in
\Z$) is an irreducible $L_{D_5 ^{(1)}} (-3\Lambda_0) \otimes
M(1)$--module. In particular
\bea
&& L_{E_{6} ^{(1)}}(-3\Lambda_0) ^{(s)} \cong  L_{D_5 ^{(1)}} (-(s+3)\Lambda_0+ s \Lambda _4 ) \otimes M(1, s) \qquad (s \ge 0),  \nonumber \\
&& L_{E_{6} ^{(1)}}(-3\Lambda_0) ^{(s)} \cong  L_{D_5 ^{(1)}}
((s-3)\Lambda_0- s \Lambda _5 ) \otimes M(1, s) \qquad (s <  0).
\nonumber \eea
\end{theorem}
\begin{proof} We omit the proof because of its similarity with the
proofs of Theorems \ref{prosta-0} and \ref{irr-charge}. The main
ingredients of the proof are the classification result from
Proposition \ref{prop-class-tilda}, the fusion rules argument (based
on Proposition \ref{tensor-D-Okada}), formula for the conformal
weight (\ref{rel-lowest-confw-D}), and the fact that $L_{E_{6}
^{(1)}}(-3 \Lambda_0)$ is simple vertex operator algebra.
\end{proof}

We obtain the decomposition:
\bea
 && L_{E_{6} ^{(1)}}(-3 \Lambda_0) = \bigoplus_{ s \in \Z
_{\geq 0}} L_{D_5 ^{(1)}} (-(s+3)\Lambda_0+ s \Lambda _4 ) \otimes
M(1, s) \nonumber \\
&& \qquad \oplus \bigoplus_{ s \in \Z _{< 0}} L_{D_5 ^{(1)}}
((s-3)\Lambda_0- s \Lambda _5 ) \otimes M(1,s). \label{decompVOADE}
\eea

As a consequence of decomposition (\ref{decompVOADE}), one also
obtains the decomposition of $L_{F_{4} ^{(1)}}(-3 \Lambda_0)$ into
direct sum of $L_{B_{4} ^{(1)}}(-3 \Lambda_0) \otimes M(1)
^+$--modules, where $M(1) ^+$ denotes the $\Z _2$-orbifold of $M(1)$
(cf. \cite{DN1}, \cite{DN2}). Namely, using conformal embeddings of
$L_{F_{4} ^{(1)}}(-3 \Lambda_0)$ into $L_{E_{6} ^{(1)}}(-3
\Lambda_0)$ and $L_{B_{4} ^{(1)}}(-3 \Lambda_0)$ into $L_{D_{5}
^{(1)}}(-3 \Lambda_0)$ from \cite{AP-ART}, one can easily obtain
that $L_{B_{4} ^{(1)}}(-3 \Lambda_0) \otimes M(1) ^+$ is a vertex
subalgebra of $L_{F_{4} ^{(1)}}(-3 \Lambda_0)$ with the same
conformal vector.

We obtain the following decomposition:
\begin{corollary} We have:
\bea && L_{F_{4} ^{(1)}}(-3 \Lambda_0)= L_{B_{4}
^{(1)}}(-3\Lambda_0) \otimes M(1) ^+ \oplus
L_{B_{4} ^{(1)}}(-4 \Lambda_0 + \Lambda_1) \otimes M(1) ^- \nonumber \\
&& \qquad \oplus  \bigoplus_{ s \in \Z _{> 0}} L_{B_4 ^{(1)}}
(-(s+3)\Lambda_0+ s \Lambda _4 ) \otimes M(1, s). \nonumber \eea
\end{corollary}

\end{document}